\documentclass[twoside,12pt]{amsart}
\usepackage{latexsym}
\usepackage{amssymb}

\makeatletter
\def\serieslogo@{}
\makeatother \makeatletter
\def\@setcopyright{}
\makeatother

\def\ep{\epsilon}

\newtheorem{thm}{Theorem}[section]

\newtheorem{lem}[thm]{Lemma}

\theoremstyle{definition}

\theoremstyle{remark}
\newtheorem{rem}{Remark}[section]

\numberwithin{equation}{section}
\begin{document}
\title[Recovery of high frequency wave fields]{Recovery of high frequency wave fields for the acoustic wave equation }
\author{Hailiang Liu and James Ralston}
\address{Iowa State University, Mathematics Department, Ames, IA 50011} \email{hliu@iastate.edu}
\address{UCLA, Mathematics Department, Los Angeles, CA 90095} \email{ralston@math.ucla.edu}

\date{June 15, 2009}

\begin{abstract} Computation of high frequency solutions to wave equations is important in many applications, and notoriously difficult in resolving wave oscillations. Gaussian beams are asymptotically valid high frequency solutions
concentrated on a single curve through the physical domain, and superposition of Gaussian beams provides a powerful tool to generate more general high frequency solutions to PDEs. An alternative way to compute Gaussian beam components such as phase, amplitude and Hessian of the phase, is to capture them in phase space
by solving Liouville type equations on uniform grids. Following \cite{LR:2009} we present a systematic construction of  asymptotic high frequency wave fields from computations in phase space for acoustic wave equations; the superposition of phase space based Gaussian beams over two moving domains is shown necessary. Moreover, we prove that the $k$-th order Gaussian beam superposition converges to the original wave field in the energy norm, at the rate of  $\ep^{\frac{k}{2}+\frac{1-n}{4}}$ in dimension $n$.
\end{abstract}

\maketitle

\bigskip

\tableofcontents

\section{Introduction}This is the continuation of our project, initiated in  \cite{LR:2009}, of developing a rigorous recovery theory for high frequency
wave fields from phase space based computations. Here we focus on
the wave equation
\begin{align}\label{pp}
Pu:=[\partial_t^2 -c(x)^2\Delta]u  =0, \quad (x, t)\in \mathbb{R}^n \times \mathbb{R},
\end{align}
where $c(x)$ is a positive smooth function, with highly oscillatory
initial data
\begin{equation}\label{ini}
u(x, 0)=A_{\rm in}(x,\ep) e^{iS_{\rm in}(x)/\ep}, \quad u_t(x,
0)=B_{\rm in}(x,\ep) e^{iS_{\rm in}(x)/\ep}.
\end{equation}
The initial phase $S_{\rm in }\in C^\infty(\mathbb{R}^n)$, and the
amplitudes $A_{\rm in}, B_{\rm in} \in C_0^\infty(\mathbb{R}^n)$
have the following asymptotic expansions:
\begin{align}\label{ae}
A_{\rm in}: &=A^{(0)}_{\rm in}(x)+\ep A^{(1)}_{\rm in}(x)+\ep^2
A^{(2)}_{\rm in}(x)+\cdots,
\\ \label{be} B_{\rm in}: &=\ep^{-1} B^{(-1)}_{\rm in}(x) +
B^{(0)}_{\rm in}(x) + \ep B^{(1)}_{\rm in}(x)+\cdots.
\end{align}
The small parameter $\ep$ represents the typical wave length of oscillations of the initial data. Propagation of oscillations of wave
 length $\ep$ causes mathematical and numerical challenges in solving high frequency wave propagation problems.

  In this article we are
 interested in the construction
of globally valid asymptotic wave fields and the analysis of their
convergence to the true solutions of the initial value problem. A
general discussion of this problem and background references are
given in the introduction to \cite{LR:2009}. We have two
objectives:
\begin{itemize}
\item[i)]  to present the construction of asymptotic solutions as superpositions over phase space;
 \item[ii)] to estimate the difference between the exact wave fields and the asymptotic ones.
\end{itemize}
The construction for (i) is based on Gaussian beams (GB) in
physical space constructed similarly to those given  for wave
equations in \cite{Tan08},
 but here the construction is carried out by solving inhomogeneous Liouville equations in phase space.
 While the result is no longer a superposition of asymptotic
 solutions to the wave equation (1.1), the superposition is nonetheless asymptotic.
  We consider {\it superpositions over two subdomains} moving with two Hamiltonian flows,
 respectively, and show that they are
asymptotic solutions by relating them to the {\it
Lagrangian superposition} through two
time-dependent symplectic changes of variables.
An argument of this type was used for the
Helmholtz equation in \cite{LQB:2007}.

For (ii), as in \cite{Tan08}, we use the well-posedness theory for
(1.1), i.e. the continuous dependence of solutions of $P\psi=f$ on
their initial data and $f$. Thus, the sources of error in the
Gaussian beam superposition for the initial value problem are the
error in approximating the initial data and the error in solving
the PDE. There are some differences between the acoustic wave
equation and the Schr\"{o}dinger wave equation. For example, the
caustics that can form are weaker.

In summary, our phase space based Gaussian beam superposition is expressed as
\begin{equation}\label{sol1}
u^\epsilon(t, y)=Z(n,\epsilon)\left[ \int_{\Omega^+(t)}u_{PGB}^+(t,
y,X)dX +\int_{\Omega^-(t)}u_{PGB}^-(t,
y,X)dX \right],
\end{equation}
where $X=(x, p)$ denotes variables in phase space $\mathbb{R}^{2n}$, $\Omega(0)$ is the domain where
 we construct initial Gaussian beams from the given data, and $\Omega^\pm(t)$ is the image of $\Omega(0)$ under the Hamiltonian flow for
  $H(x, p)=\pm c(x)|p|$. The functions $u_{PGB}^\pm(t,y,X)$ are constructed using the phase space based Gaussian beam {\it Ansatz},
  and $Z(n ,\ep) \sim \ep^{-n/2}$ is a normalization parameter. Our result shows that for the $k-$th order
  phase space Gaussian beam superposition, the following estimate holds
\begin{equation}\label{est1}
\|(u^\epsilon -u)(t, \cdot)\|_{E} \lesssim\|(u^\epsilon(0,
\cdot)-u_{\rm in}(\cdot) \|_{E} + |\Omega(0)|
\ep^{\frac{k}{2}+\frac{1-n}{4}},
\end{equation}
where $\|e\|^2_{E}:=\frac{\ep^2}{2}\int_{\mathbb{R}^n}
[c^{-2}|e_t|^2 +|\nabla_x e|^2]dx$. Here and in what follows we use
$A \lesssim B$ to denote the estimate $A\leq CB$ for a constant $C$
which is independent of  $\ep$.

For the initial data of the form $(A_{\rm in}(x,\ep), B_{\rm
in}(x,\ep))e^{iS_{\rm in}(x)/\ep}$ we need a superposition over an
n-dimensional submanifold of phase space. The asymptotic solution
is then represented as
\begin{equation}\label{sol2}
u^\epsilon(t, y)=Z(n,\epsilon)\left[ \int_{\Omega^+(t)}u_{PGB}^+\delta(w^+)dX +\int_{\Omega^-(t)}u_{PGB}^-\delta(w^-)dX \right],
\end{equation}
where $w^\pm$ is obtained from the Liouville equation
$$
\partial_t w+H_p\cdot \nabla_x w-H_x \cdot \nabla_p w=0, \quad w(0, X)=p-\nabla_x S_{\rm in}(x),
$$
with $H(x, p)=\pm c(x)|p|$, respectively. Our result shows that
\begin{equation}\label{est2}
\|(u^\epsilon -u)(t, \cdot)\|_{E} \lesssim  \ep^{\frac{k}{2}+\frac{1-n}{4}}.
\end{equation}
Here the exponent $k/2$ reflects the accuracy of the Gaussian beam in solving the PDE. It will increase when one uses more accurate beams.
The exponent $\frac{1-n}{4}$ indicates the damage done by the caustics.

We now conclude this section by outlining the rest of this paper:
in Section 2 we start with Gaussian beam solutions in physical
space, and define the phase space based GB {\it Ansatz}  through
the Hamiltonian map. Section 3 is devoted to a recovery scheme
through superpositions over two moving domains.  The total error is
shown to be bounded by an initial error and the evolution error of
order $\ep^{(3-n)/4}$. Control of initial error is discussed in
Section 4.  Convergence rates are obtained for first order GB
solutions in Section 5. In Section 6 we present an example to
illustrate these constructions. Extensions to higher order GB
approximations are given in Section 7.

\section{Phase space based Gaussian beam Ansatz}
As is well known, the idea underlying Gaussian
beams \cite{Ralston05} is to build asymptotic
solutions concentrated on a single ray path in
$\mathbb{R}_t\times \mathbb{R}^n_x$. This means
that, given a ray path $\gamma$ parameterized by
$(t,x(t))$, one makes the ansatz
\begin{equation}\label{an}
u^\epsilon(t, y)=A(t, y,\ep)e^{i\Phi(t,
y)/\epsilon},
\end{equation}
where $\Phi(t, x(t))$ is real, and $Im\{\Phi(t, y)\} > 0$ for $y
\not= x(t)$. The amplitude is allowed to be complex and has an
asymptotic expansion in terms of $\ep$:
$$
A(t, y,\ep)=A_0(t, y) +\ep A_1(t, y)+\cdots+\ep^N
A_N(t, y).
$$
We wish to build asymptotic solutions to $Pu(t, y) = 0$, i.e., we
want $P u^\epsilon = O(\epsilon^{M})$. Substituting from (\ref{an})
and grouping terms multiplied by the same power of $\ep$, we obtain
the equations of geometric optics:
\begin{equation}\label{pap}
   P[A(t, y,\ep)e^{i\Phi(t, y)/\ep}]
   =\left( \sum_{j=-2}^{N} c_j(t, y)\ep^j \right)e^{i\Phi(t, y)/\ep},
\end{equation}
where for $ G(t, y)=|\partial_t \Phi|^2 -c^2|\nabla_y \Phi|^2$,
\begin{align*}
c_{-2}(t, y) &= - G(t, y)A_0,\\
c_{-1}(t, y) &=2i LA_0 +  G(t, y) A_1,\\
c_{l-1}(t, y) &=2i L A_{l}+ G(t, y)A_{l+1} + P[A_{l-1}], \quad l=1,
\cdots N-1.
\end{align*}
Here $L$  is the linear differential operator,
$$
L= \Phi_t \partial_t -c^2 \nabla_y \Phi \cdot \nabla_y   +
\frac{1}{2}P[\Phi].
$$
Since  $e^{i\Phi/\ep}$ decays rapidly away from $\gamma$, to make
$P(Ae^{i\Phi/\ep})= O(\ep^M)$ for a given $M\in \mathbb{Z}$, we
only need to make $c_j$ vanish on $\gamma$ to sufficiently high
order. In this work we discuss mainly the lowest order Gaussian
beam solutions, followed by an extension to higher order Gaussian
beam superpositions in Section 7.

We begin with $c_{-2}=0$, i.e. $ G=0$. This leads to two eikonal
equations
\begin{equation} \label{so}
\partial_t \Phi +H(x, \nabla_x \Phi)=0, \quad H(x, p)=\pm c(x)|p|.
\end{equation}
The leading amplitude solves
\begin{equation}\label{aa}
\partial_t A + H_p\cdot \nabla_x A= \frac{AP[\Phi]}{2H(x, \nabla_x \Phi)}.
\end{equation}
We continue to denote the phase space variable as $X=(x, p)$, and
let  $X_0=(x_0, p_0)$ denote the initial state. Then the equations
for the bicharacteristics  $X=X^\pm (t, X_0)$ originating from
$X_0$ at $t = 0$ are
\begin{equation}\label{x}
    \frac{d}{dt}X(t, X_0)=V(X(t, X_0)), \quad X(0, X_0)=X_0.
\end{equation}
The vector field $V=(H_p, -H_x)$ is divergence free, and hence this
flow preserves the volume on phase space.

From now on we include the initial data $X_0$ as
a parameter in the phase:  $\Phi=\Phi(t, y; X_0)$
and the amplitude:  $A=A(t, y; X_0)$. We apply
Taylor expansion of the phase $\Phi$ and the
amplitude $A$ about $x=x(t, X_0)$ to obtain
\begin{equation}\label{phi0}
  \Phi(t,y; X_0)=S(t; X_0)+p(t, X_0)(y-x(t, X_0))
+\frac{1}{2}(y-x(t, X_0))^\top M(t;X_0)(y-x(t, X_0)),
\end{equation}
with $p(t, X_0)=\partial_y \Phi(t, x(t, X_0); X_0)$ and
$$
S(t;X_0)=\Phi(t,x(t, X_0); X_0), \quad M(t;X_0)=\partial_y^2
\Phi(t, x(t, X_0); X_0).
$$
For the amplitude we set $A(t, y; X_0)=A(t;X_0)$  with $A(t; X_0)=A(t,x(t,
X_0);X_0).$  Then we get the equations along the curve $\gamma$ for $S$
\begin{align}\label{s}
\frac{d}{dt} S(t; X_0)=0, \quad S(0; X_0)=S_{\rm in}(x_0),
\end{align}
and the Hessian $M$
\begin{equation}\label{m}
\frac{d}{dt}M(t; X_0) + H_{xx}+H_{xp}M +MH_{px}+ MH_{pp}M=0, \quad
 M(0; X_0)=M_{\rm in}(x_0).
\end{equation}
Using the eikonal equation $\partial_t \Phi +H(x, \nabla \Phi)=0$
twice, we see that
$$
P[\Phi]=\partial_t[-H(x, \nabla \Phi)]-c^2\Delta \Phi =H_p\cdot H_x +H_pMH_p -c^2 Tr(M).
$$
This with (\ref{aa}) shows that the amplitude along the ray, $
A(t;X_0)$, satisfies
\begin{equation}
\label{a}
\frac{d}{dt} A(t;X_0)  =\frac{A}{2H}\left[ H_p\cdot H_x +H_pMH_p -c^2 Tr(M)\right], \quad A(0; X_0)=A_{\rm in}(x_0).
\end{equation}
We have introduced this form of the transport equation because it
is easier to translate to  Eulerian coordinates. The essential idea
behind the Gaussian beam method is to choose some complex Hessian
$M_{\rm in}$ initially so that $M$ remains bounded for all time,
and its imaginary part is positive definite. Equation (\ref{a})
shows that the amplitude $A(t;X_0)$ will also remain bounded for
all time.

The above construction ensures that the following
GB {\it Ansatz} is an approximate solution
$$
u_{GB}(t, y; X_0)= u_{GB}^+(t, y;X_0) + u_{GB}^-(t, y;X_0),
$$
where
$$
u_{GB}^\pm (t, y;X_0)=  A^\pm (t;X_0)\exp\left(\frac{i}{\epsilon}\Phi^\pm(t,
y;X_0)\right),
$$
where both  $A^\pm (t;X_0)$ and $\Phi^\pm(t, y;X_0)$ are computed from (\ref{a}) and (\ref{phi0}) with $H=\pm c(x)|p|$, respectively.

Here $A^\pm (0; X)$ are to be chosen so that a superposition will
match the initial data
$$
(u, u_t)|_{t=0}=(A_{\rm in}, B_{\rm
in})e^{iS_{\rm in}/\ep}
$$
to leading order.  For this matching we need (for
$X=(x,\nabla S_{\rm in}(x))$)
\begin{align*}
& A^+(0;X)+ A^-(0;X)=A^{(0)}_{\rm in}(x),\\
& \frac{i}{\ep} A^+(0;X) \partial_t \Phi^+(0, x;
X) + \frac{i}{\ep}A^-(0;X)\partial_t \Phi^-(0, x;
X) =\frac{1}{\ep} B^{(-1)}_{\rm in}(x).
\end{align*}
In the second relation we took only the leading term in
$e^{-iS_{\rm in}/\ep}u_t$.  Since the two Hamiltonians have
different signs,
$$\Phi^\pm(0, x; X)=S_{\rm in}(x) \quad {\rm and}\quad  \partial_t
\Phi^\pm(0, x; X)=\mp c(x)|\nabla_x S_{\rm in}(x)|,$$  the second
relation gives
\begin{equation}\label{id}
A^+(0; X) - A^-(0; X)=\frac{i B^{(-1)}_{\rm
in}(x)}{c(x) |\nabla_x S_{\rm in}(x)|}.
\end{equation}
Hence solving for $A^\pm$ we have
\begin{equation}\label{apm}
    A^\pm(0;X)=\frac{1}{2} \left( A^{(0)}_{\rm in}(x) \pm \frac{i B^{(-1)}_{\rm in}(x)}{c(x)| \nabla_x S_{\rm in}(x)|} \right).
\end{equation}
Note that we could simplify the superposition by taking some special initial data such that $B^{(-1)}_{\rm in}(x)= - iA^{(0)}_{\rm in}(x)c(x)|\nabla S_{\rm in}(x)|$. The advantage of these special choices is that we do not need a sum of two Gaussians to approximate the solution. We also note that for given initial $B^\ep_{\rm in}$ of order $O(1)$, i.e., $B^{(-1)}_{\rm in}=0$, we see that
$A^\pm(0;X)=\frac{1}{2}  A^{(0)}_{\rm in}(x).$
\section{Recovery of the high frequency wave fields}
Since the wave equation we consider is linear, the high frequency wave
field $u$ at $(t, y)$ in physical space is expected to be generated by a superposition of neighboring Gaussian beams
\begin{equation}\label{GB+}
u^\epsilon(t, y)=Z(n,\epsilon)\int_{\Omega(0)}u_{GB}(t,
y;X_0)dX_0,
\end{equation}
where $\Omega(0)$ is a bounded open set
containing
$$
\{X_0: \quad x_0\in {\rm supp}(A_{\rm
in})\cup {\rm supp}(B_{\rm in}), \quad
p_0 \in {\rm range}(\partial_x S_{\rm in})\}.
$$
 The normalization parameter
$Z(n, \epsilon)\sim \ep^{-n/2}$ is determined by
matching initial data against the Gaussian
profile.

Since the flows $X^\pm(t;X_0)$ are volume
preserving in phase space,
$$det\left(\frac{\partial X^\pm(t,
X_0)}{\partial{X_0}}\right)=1.$$ Using
$X=X^\pm(t, X_0)$ and their  inverses
$X_0=X_0^\pm(t, X)$, we obtain our  Gaussian beam
{\it Ansatz} in phase space
$$
u_{PGB}^\pm (t, y, X):=u^\pm_{GB}(t, y; X_0^\pm(t, X)).
$$
From (\ref{GB+}) it follows that
\begin{align} \notag
u^\epsilon(t, y)& =Z(n,\epsilon)\int_{\Omega(0)}\left[  u_{GB}^+(t,
y;X_0) + u_{GB}^-(t,
y;X_0) \right]dX_0 \\ \label{egb}
&= Z(n,\epsilon) \left[ \int_{\Omega^+(t)} u_{PGB}^+(t,
y, X) dX + \int_{\Omega^-(t)} u_{PGB}^-(t,
y, X) dX \right],
\end{align}
where
$$
\Omega^\pm(t)=X^\pm(t, \Omega(0)).
$$
Each phase space Gaussian beam has the form
\begin{equation}\label{pgb}
    u_{PGB} (t, y, X)= \tilde A (t, X)\exp \left( \frac{i}{\epsilon}
\tilde \Phi (t, y, X) \right),
\end{equation}
where
\begin{equation}\label{tphi}
\Phi (t, y, X)=\tilde S(t, X)+p\cdot (y-x)
+\frac{1}{2}(y-x)^\top \tilde M (t, X)(y-x).
\end{equation}
Note that though $u_{PGB}^\pm (t, y, X)$ are no
longer asymptotic solutions of the wave equation
in $(t,y)$, their superpositions over the moving
domains $\Omega^\pm(t)$ in $X$ remain asymptotic
solutions.

Let $\mathcal{L}$ be the Liouville operator defined by
\begin{equation}\label{liou}
    \mathcal{L} := \partial_t +V\cdot \nabla_X.
\end{equation}
If $\tilde w(t, X)$ is the phase space representative of  $w(t;X_0)$ in the sense that $w(t;X_0)=w(t, X(t, X_0))$
for any $t>0$, then
$$
\frac{d}{dt}w(t; X_0) = \mathcal{L} \tilde w(t, X).
$$
Hence from the Lagrangian formulation of equations
for $(S, M, A)$ we obtain PDEs for $(\tilde S, \tilde M, \tilde A)$ in (\ref{s}), (\ref{m}) and (\ref{a}):
\begin{align}\label{sp}
& \mathcal{L}(\tilde S) =0, \quad \tilde S(0,
X)=S_{\rm in}(x),\\\label{mp}
 &
 \mathcal{L}(\tilde M) + H_{xx}+H_{xp}\tilde M +\tilde MH_{px}+ \tilde MH_{pp} \tilde M=0, \quad
 \tilde M(0, X)=M_{\rm in}(x),\\ \label{ap}
& \mathcal{L}(\tilde A)  =\frac{\tilde A}{2H}\left[ H_p\cdot H_x +H_p \tilde MH_p -c^2 Tr(\tilde M)\right], \quad \tilde A(0, X)=A_{\rm in}(x),
\end{align}
where $H(x, p)=c(x)|p|$ or $H(x, p)=-c(x)|p|$.
The heart of the matter is equation (\ref{mp}).
It is known from \cite{Ralston82} that, if
$M_{\rm in}$ is symmetric and the imaginary part
of $M_{\rm in}$
is positive definite, then a global solution $\tilde M$ to (\ref{mp}) is guaranteed and has the properties:\\
i) $\tilde M=\tilde M^T,$ and \\
ii) $Im(\tilde M)$ is positive definite for all
$t>0$.

There are several ways of computing $\tilde M$.
Following \cite{JWY:2008} (see also \cite[Section
7]{LR:2009}), we use a level set method to
construct the Hessian:
\begin{equation}\label{M}
  \tilde M=-g_x (g_p)^{-1},
\end{equation}
where $g=\phi_1(t, X)+i\phi_2(t, X)$ with $\phi_i$ obtained by solving the Liouville equation
$$
 \mathcal{L}(\phi)=0.
$$
From the well-posedness theory of the wave equation
 we have the following.
\begin{lem}\label{lem31}
Let $u$ satisfy $P[u]=0$ in $[0, T]\times
\mathbb{R}^n$ with $(u, u_t)$ given at $t=0$, and
let $u^\epsilon$ be an asymptotic solution. Then
the error $e=u^\ep-u$ satisfies
\begin{equation}\label{e0}
   \|e(t)\|_E \leq \|e(0)\|_E  + \ep \int_0^t \left\|c^{-1} P[u^\ep] \right\|_{L^2}d \tau,
\end{equation}
where $\|e\|_E =\sqrt{2E}$ and
$$
E: =\frac{\ep^2}{2} \int_{\mathbb{R}^n} \left[c^{-2}|e_t|^2 +|\nabla_x e|^2 \right]dx.
$$
\end{lem}
\begin{proof}
Since we start with the data with compact
support, at any finite time the support of the
solution remains bounded (due to finite speed of
propagation for the wave equation).

Let $e=u^\ep-u$. Then from $P[u]=0$
$$
P[e]=P[u^\ep]-P[u]=P[u^\ep].
$$
We now have
\begin{align*}
\frac{d}{dt}E(t) & = \ep^2 \int_{R^n}  \left[ c^{-2} e_t e_{tt} +\nabla e \cdot \nabla e_t \right] dx \\
& = \ep^2 \int_{R^n} \left[ \nabla \cdot(e_t\nabla e)+ c^{-2} e_t P[u^\ep] \right]dx \\
& \leq \ep^2  \left\| c^{-1} e_t \right\|_{L^2} \left\|c^{-1} P[u^\ep] \right\|_{L^2}\leq \ep \sqrt{2E}  \left\|c^{-1} P[u^\ep]\right\|_{L^2}.
\end{align*}
This, upon integration in time, leads to the desired estimate.
\end{proof}

\section{Control of initial error}
For the initial phase $S_{\rm in}$, we set
$p_0=\nabla_x S_{\rm in}(x_0)$  and form the
Lagrangian superpositions
$$
u^\ep(t, y)= Z(n,\epsilon)\int_{\Omega(0)}u_{GB}(t,
y; X_0)\delta(p_0-\nabla_x S_{\rm in}(x_0))dX_0.
$$
In order to track the deformation of the surface $p-\nabla_x S_{\rm in}(x)=0$ as time evolves, we introduce two level set functions $w=w^\pm (t, X)$ such that
\begin{equation}\label{w}
\mathcal{L}[w]=0, \quad w(0, X)=p-\nabla_x S_{\rm in}(x),
\end{equation}
with $H(x, p)=\pm c(x)|p|$. Here $w^\pm$ gives $\phi_2$ needed in (\ref{M}) and $\phi_1$ can be obtained  from solving the respective Liouville equation with $\phi_1(0, X)=x$.

Using the volume preserving maps $X=X^\pm(t,
X_0)$, leads to the Gaussian beam superposition
in phase space
\begin{equation}\label{kegb}
u^\ep(t, y)=Z(n,\epsilon)\left[ \int_{\Omega^+(t)} u_{PGB}^+ \delta(w^+)dX+ \int_{\Omega^-(t)}u_{PGB}^-
\delta(w^-)dX \right],
\end{equation}
where $ \Omega^\pm(t)=X^\pm(t, \Omega(0)). $  Our
choice of initial data for the beams in this
superposition will be made to match the initial
data in (\ref{ini}). Set
\begin{equation}\label{i0}
    I(0)=\{x:\quad (x, p)\in \Omega(0), \quad p=\nabla_x S_{\rm in}(x)\}.
\end{equation}
We now use the Lagrangian formulation of the GB superposition to match the initial data.
\begin{equation}\label{uep}
    u^\ep(t, y)=Z(n,\epsilon)\int_{I(0)}u_{GB}(t,y; x_0)dx_0.
\end{equation}
Here and in what follows we use $u_{GB}(t,y; x_0)$ for $u_{GB}(t,y; x_0, \nabla_x S_{\rm in}(x_0))$.
If we take $S^\pm(0;x_0)=S_{\rm in}(x_0)$,  $M^\pm(0; x_0)=\partial_x^2 S_{\rm in}(x_0)+i\beta I$ with $\beta>0$ as well as $A^\pm(0; X_0)$ as defined in  (\ref{apm}), then
$$
u^\ep(0, y) = Z(n,\epsilon)\int_{I(0)}A_{\rm in}^{(0)}(x_0)e^{i\Phi(0, y;x_0)/\ep}dx_0,
$$
where
$$
\Phi(0, y; x_0)= T_{2}^{x_0}[S_{\rm in}](y) -\frac{\beta}{2}|y-x_0|^2.
$$
Here $T_j^x[S](y)$ denotes the $j^{th}$ order
Taylor polynomial of $S$ about $x$ at the point
$y$. Setting
$$
Z(n, \ep)=\left(\frac{\beta}{2\pi \ep} \right)^{n/2},
$$
we have
$$
u^\ep(0, y) = \int_{I(0)}A_{\rm in}^{(0)}(x_0)e^{\frac{i}{\epsilon}\left[T_{2}^{x_0}[S_{\rm in}](y)\right]} K\left(x_0-y, \frac{\epsilon}{2\beta}\right)dx_0,
$$
where $K(x, \tau)=\frac{1}{(4\pi \tau)^{n/2}}e^{-\frac{|x|^2}{4\tau}}$
is the usual heat kernel, satisfying ${\rm limit}_{\tau \downarrow
0}K(x, \tau)=\delta (x)$ as distributions on $\mathbb{R}^n$, and
$$
\int_x K(x-y, \tau)dx=1, \quad \forall \tau>0, \; y\in \mathbb{R}^n.
$$
On the other hand the initial wave field is
$$
u(0, y) =A_{\rm in}^\ep (y)e^{iS_{\rm in}(y)/\epsilon}=\int_{\mathbb{R}^n}
A_{\rm in}^\ep (y)e^{iS_{\rm in}(y)/\epsilon}K\left(x-y, \frac{\epsilon}{2 \beta}\right)dx.
$$
Both the phase and amplitude in the integrand can
be approximated by their Taylor expansion when
$|x-y|$ is small, say  $|x-y|<\epsilon^{1/3}$,
and the integral over the complement of this
neighborhood will then be $O(\exp
(-c\epsilon^{-1/3}))$ for some $c>0$. Thus the
main contributions to the error come from the
remainder terms in the Taylor expansions, and
this leads to
\begin{lem} \label{tan}\cite{Tan08}
Let $S_{\rm in} \in C^\infty (R^n)$ be a real-valued function, and $A_{\rm in} \in
C_0^\infty (R^n) $.
Then
\begin{align}\label{l2}
 \left\|u(0, \cdot)-u^\ep(0, \cdot) \right\|_{L^2} & \lesssim
\ep^{\frac{1}{2}}, \\ \label{h1}
\ep \left\|u(0, \cdot)-u^\ep(0, \cdot) \right\|_{H^1} & \lesssim
\ep^{\frac{1}{2}}.
\end{align}
\end{lem}
\begin{rem}We note that a cutoff function is necessary and important when one is building beams of higher accuracy 
\end{rem}


We now show the initial error of time derivative of the GB superposition is also under control.
We compute the time derivative of (\ref{uep}) to obtain
$$
\partial_t u^\ep(t, y) =Z(n,\epsilon)\int_{I(0)} \partial_t u_{GB}(t,
y; x_0)dx_0,
$$
where $u_{GB}(t,y; x_0)=u_{GB}^+(t,y; x_0)+u_{GB}^-(t,y; x_0)$ with
\begin{align*}
\partial_t u_{GB}^\pm(t,y; x_0)=\left[\partial_t A^\pm +\frac{i}{\ep}A^\pm \partial_t \Phi^\pm \right]e^{i\Phi^\pm(t, y;x_0)/\ep}.
\end{align*}
Note that the GB construction ensures that
\begin{align*}
\partial_t \Phi^\pm(t, y;x_0)= \mp c(y)| \nabla \Phi^\pm(t,y;x_0)| +O(|y-x(t, x_0)|^3).
\end{align*}
Recall (\ref{a}) we have  $ \partial_t A(t;x_0) \sim
O(1). $ Hence from (\ref{id}) we have
$$
\partial_t u_{GB}(0,y; x_0)=\left[O(1)+ \ep^{-1} (B^{(-1)}(x_0)+O(|y-x_0|))\right]e^{i\Phi^\pm(0,
y;x_0)/\ep}.
$$
Note that $$ \left\| Z(n, \ep) \int_{I(0)}
\left[O(1)+O\left(\frac{|y-x_0|}{\ep}\right)\right]e^{i\Phi^\pm(0,
y;x_0)/\ep}dx_0 \right\|_{L^2_y}\leq C(1+\ep^{-1/2}),
$$
which together with Lemma \ref{tan} again gives
 \begin{align}\label{e0}
\|u^\ep(0,\cdot)- u(0, \cdot)\|_{E}\leq  \ep \|u^\ep(0,\cdot)- u(0, \cdot)\|_{H^1}+\ep  \|\partial_t u^\ep(0,\cdot)-\partial_t u(0, \cdot)\|_{L^2_y} \lesssim \ep^{1/2}.
\end{align}
\begin{rem}
The above analysis shows that one could choose $B_{\rm in}^{(-1)}$ to simplify the superposition, as was pointed out
 in Section 2. For example,  \\
(i) for $B_{\rm in}^{(-1)}=-ic(x)| \nabla S_{\rm
in}|$, then $A^+(0;X_0)=A_{\rm in}^{(0)}(x_0),
\quad A^-(0;X_0)=0 $
\begin{equation}\label{kegb++}
u^\ep(t, y)=Z(n,\epsilon)\left[ \int_{\Omega^+(t)}u_{PGB}^+ \delta(w^+)dX \right];
\end{equation}
(ii) for $B_{\rm in}^{(-1)}=0$, then
$A^\pm(0;X_0) =\frac{1}{2}A_{\rm in}^{(0)}(x_0)$
\begin{equation}\label{kegb+}
u^\ep(t, y)=Z(n,\epsilon)\left[ \int_{\Omega^+(t)}u_{PGB}^+ \delta(w^+)dX+ \int_{\Omega^-(t)}u_{PGB}^-
\delta(w^-)dX \right].
\end{equation}
\end{rem}

\section{Propagation of the approximation error}
We now turn to quantify the evolution error $P[u^\ep]$. Recall the Schur's lemma: If $[Tf](y) =\int
K(x, y)f(x)dx$ and
$$
{\rm sup}_x
\int_y |K(x, y)|dy = C_1, \; {\rm sup}_y
\int_x |K(x, y)|dx = C_2,
$$
then
$$
\|Tf\|_{L^2}\leq \sqrt{C_1C_2}\|f\|_{L^2}.
$$
\begin{proof}
We have by Schwartz
\begin{align*}
|[Tf](y)|^2 \leq
\left(
\int|K(x,y)|f(x)dx\right)^2 & \leq
\int |K(x,y)|dx\int
|K(x,y)||f(x)|^2dx\\
 & \leq  C_2 \int
|K(x,y)||f(x)|^2dx.
\end{align*}
So integrating both sides in $y$ and taking the square root gives the result.
\end{proof}
We now apply Schur's lemma to a typical term in $\int_{I(0)} P[u^\ep]dx_0$:
$$
[TA](y) = \int_{I(0)} A(t; x_0) F(t, y; x_0)e^{i\Phi(t, y;x_0)/\ep}dx_0,
$$
where the imaginary part of $\Phi(t, y;x_0)$ is bounded below by $cI$ and for convenience we will assume that $|F|\leq |y-x(t, x_0)|^k$. Then one can apply Schur's lemma with
$$
C_1={\rm sup}_{x_0}\int_{\mathbb{R}^n} |y-x(t, x_0)|^k e^{-(c/\ep)|y-x(t, x_0)|^2}dy=\ep^{\frac{k}{2}+\frac{n}{2}}\int_z |z|^ke^{-c|z|^2}dz, \quad {\rm and}
$$
\begin{equation}\label{c2}
C_2(t, \ep)={\rm sup}_{y}\int_{I(0)}|y-x(t, x_0)|^k e^{-(c/\ep)|y-x(t, x_0)|^2}dx_0.
\end{equation}
In general one does not know what $C_2(t, \ep)$ will be. As long as $A$ has compact support $C_2$ will be at least
bounded by $c\ep^{k/2}$.  Thus the error in $L^2$ norm will be bounded by $c\ep^{k/2 +n/4}$. We now show that for the wave equation, a better rate can be obtained.
\begin{lem}We have
$$
C_2(t, \ep) \lesssim \ep^{(k+1)/2}.
$$
\end{lem}
\begin{proof} From (\ref{s}) and taking $p_0=\nabla_x S_{\rm in}(x_0)$ it follows
$$
S(t, x(t, x_0))=S_{\rm in}(x_0), \quad \forall t>0.
$$
Differentiation of this equation in $x_0$ gives
$$
\frac{\partial x}{\partial x_0}p=p_0, \quad p(t, x_0):=\nabla_x S(t, x(t, x_0)).
$$
For non-constant initial phase, at least one element in the deformation matrix $\frac{\partial x}{\partial x_0}$ is non-zero. Assume $\frac{\partial x_1}{\partial x_{01}}\not=0$ near $x_0^*$, then writing $x_0=(x_{01}, \hat x_0)$ there exists a function $h$ such that $x_{01}=h(t, z, \hat x_0) $ and
 $$
z\equiv x_1(t, h(t, z, \hat x_0), \hat x_0)
$$
in the neighborhood of $x_0^*$.  Also the map $(x_{01}=h(t, z, \hat x_0), \hat x_0) \to (z, \hat x_0)$ is invertible, with the Jacobian determined by
$$
J=\det\left( \frac{\partial (x_{01}, \hat x_0) }{\partial(z, \hat x_0)}\right)=\left|\frac{\partial h}{\partial z}\right|=\left|\frac{\partial x_1}{\partial x_{01}}\right|^{-1}.
$$
With this map we rewrite the underlying quantity as
$$
C_2= \int_{(z, \hat x_0)}(|\hat y-\hat x(t, z, \hat x_0)|^2+|y_1-z|^2)^{k/2}\exp \left( -\frac{c}{\ep}(|\hat y-\hat x(t, z, \hat x_0)|^2+|y_1-z|^2) \right)J d \hat x_0 dz.
$$
Using a stretched coordinate in $z$ so that $z-y_1=\sqrt{\ep}\xi$, with $a:=\hat y-\hat x(t, z, \hat x_0)$, we obtain
\begin{align*}
C_2 &= \sqrt{\ep}\int_{(\xi, \hat x_0)}(|a|^2+\ep|\xi|^2 )^{k/2}e^{-c|\xi|^2} \exp \left( -\frac{c}{\ep}|a|^2
\right)J d \hat x_0 d\xi.
\end{align*}
Rewriting $e^{-c|\xi|^2}=e^{-c|\xi|^2/2}\cdot  e^{-c|\xi|^2/2}$, and using the fact that $e^{-c|\xi|^2/2} \leq 1$
and  $|\xi|^2e^{-c|\xi|^2/2} \leq C$, we obtain
\begin{align*}
C_2 & \leq \sqrt{\ep}\int_{(\xi, \hat x_0)} (|a|^2 +C\ep)^{k/2} e^{-c|\xi|^2/2}e^{-c|a|^2/\ep }Jd \hat x_0 d\xi \quad \\
& \lesssim \sqrt{\ep} \ep^{k/2} \int_{(\xi, \hat x_0)}e^{-c|\xi|^2/2}Jd\hat x_0 d\xi.
  \end{align*}
Here we have used the fact that $(|a|^2 +C\ep)^{k/2}e^{-c|a|^2/\ep } \lesssim \ep^{k/2}$ for any
$a\in \mathbb{R}^{n-1}$. As long as the initial domain for $x_0$ is finitely compact, the above integral is uniformly bounded. Note that the local feature of the used map is not restricted, since one could use a partition of unity to decompose $C_2$ into a finite sum of terms with the same rate of error. The desired estimate thus follows.
\end{proof}
This lemma enables us to conclude the following key estimate
\begin{equation}\label{Tf}
\|T[A]\|_{L^2} \lesssim \ep^{k/2 +(1+n)/4},
\end{equation}
which will be used to prove the following theorem.
\begin{thm}\label{3.3} Let $P=\partial_t^2 -c^2(x)\Delta $ be the linear wave operator and $u^\ep$ be defined in (\ref{kegb})
with $Im(M^\pm _{\rm in})=\beta I$ and $Z(n, \epsilon)=(\beta/(2\pi\ep))^{n/2}$. If both $A_{\rm in}$ and $B_{\rm in}$ have compact supports,  then $u^\epsilon$ is an asymptotic solution and
satisfies
\begin{equation}\label{pk}
   \|P[u^\ep](t, \cdot)\|_{L_y^2} \lesssim
\ep^{-\frac{1+n}{4}}.
\end{equation}
\end{thm}
\begin{proof}
Using the volume-preserving map of $X=X(t, X_0)$ and $w(t, X(t, X_0))=w(0, X_0)$, we obtain
\begin{align*}
u^\ep(t, y)&  =Z(n,\epsilon)\int_{\Omega(0)}u_{PGB}(t,
y, X(t, X_0))\delta(w(t, X(t, X_0)))dX_0\\
& =Z(n,\epsilon)\int_{\Omega(0)}u_{GB}(t,
y; X_0)\delta(w(0, X_0))dX_0\\
&= Z(n,\epsilon)\int_{\Omega(0)}u_{GB}(t,
y; X_0)\delta(p_0-\nabla_x S_{\rm in}(x_0))dX_0 \\
& =Z(n,\epsilon)\int_{I(0)}u_{GB}(t,
y; x_0)dx_0.
\end{align*}
According to the GB construction, $u_{GB}(t,
y; x_0)$ is an asymptotic solution for each $x_0$, so will be their superpositions $u^\epsilon(t, y)$.
It remains to verify (\ref{pk}).  First we see that
$$
P[u^\epsilon(t, y)]=Z(n,\epsilon)\int_{I(0)}P[u_{GB}(t,
y; x_0)]dx_0,
$$
where
\begin{equation}
   P[A(t;x_0)e^{i\Phi(t, y;x_0)/\ep}]
   =\left( \ep^{-2} c_{-2}(t, y) +\ep^{-1}c_{-1} +c_0 \right)e^{i\Phi(t, y;x_0)/\ep},
\end{equation}
where for $G(t, y)=|\partial_t \Phi|^2
-c^2|\nabla_y \Phi|^2$, we have
\begin{align*}
c_{-2}(t, y) &= - G(t, y)A,\\
c_{-1}(t, y) &=2i \left[ \partial_t A \partial_t\Phi +\frac{1}{2}AP[\Phi] \right],\\
c_{0}(t, y) &=\partial_t^2 A(t;x_0).
\end{align*}
 Using Taylor expansion around $x=x(t, x_0)$ we have
$$
G(t, y)=G(t, x)+\partial_x G(t, x)\cdot(y-x) +\frac{1}{2}(y-x)^\top \partial_x^2G (y-x)+O(|y-x|^3).
$$
Then the Gaussian beam construction sketched in Section 2 ensures that
$$
|c_{-2}(t, y)|\leq C|A||y-x|^{3}.
$$
Also using the construction for $A$ , we are able to show
$$
|c_{-1}(t, y)|\leq C|A||y-x|, \quad |c_{0}(t, y)|\leq C|A|.
$$
The construction with positive $Im(M)$ guarantees that
$$
\Phi(t, y;x_0)\geq c |y-x|^2.
$$
Consequently,
\begin{align*}
 Z^{-1} \|P[u^\epsilon(t, \cdot)]\|_{L^2} & \leq  \left \|
\int_{I(0)} A e^{-Im(\Phi(t,y;x_0))/\ep}\left|\ep^{-2} c_{-2} +\ep^{-1} c_{-1}  + c_0 \right|dx_0  \right\|_{L^2_y} \\
& \leq  \sum_{j=-2}^0 \ep^{j} \left \|
\int_{I(0)} |A||c_j| e^{-c|y-x(t,
x_0)|^2/\epsilon} dx_0\right\|_{L^2_y},
\end{align*}
continuing the estimate by using  the key estimate (\ref{Tf}) with $k=3, 1, 0$ for $F=c_{-2}, c_{-1}, c_0$, respectively
\begin{align*}
&  \lesssim \left[\ep^{-2} \ep^{3/2} +\ep^{-1}\cdot \ep ^{1/2}+1 \right]\ep^{(1+n)/4}  \\
&  \lesssim \ep^{-1/2 +(1+n)/4},
\end{align*}
which when using $Z\sim \ep^{-n/2}$ proves the result.
\end{proof}
This combined with the obtained initial error and total error estimate in Lemma 3.1 gives
\begin{thm}\label{thm4.3}
Given $T>0$, and let $u$ be the solution of the wave equation subject to the initial data $(u, u_t)(0, x)=(A_{\rm }^\ep, B_{\rm in}) e^{iS_{\rm in}(x)/\ep}$. Let $u^\epsilon$
be the first order approximation defined in (\ref{kegb}) with initial data
satisfying $S^\pm(0; x)=S_{\rm in}(x)$,  $ M^\pm(0;x) =
\partial_x^2 S_{\rm in}(x)+i\beta I$, and $ A^\pm(0; x)=\frac{1}{2} \left( A^{(0)}_{\rm in}(x)
 \pm \frac{ iB^{(-1)}_{\rm in}}{c(x)| \nabla_x S_{\rm in}|} \right)$  with $|supp(A^\ep_{\rm in})|+|supp(B^\ep_{\rm in}) |<\infty$.
Then there exists $\ep_0>0$, a normalization parameter $Z(n,
\epsilon) =\left( \frac{\beta}{2\pi \ep} \right)^{n/2}$, and a constant $C$ such that for all $\epsilon \in (0,
\epsilon_0)$
\begin{equation}\label{eoe}
\|(u^\epsilon -u)(t, \cdot)\|_{E} \leq C \ep^{\frac{1}{2} +\frac{1-n}{4}}
\end{equation}
for $t\in [0, T]$.
\end{thm}

\section{An example}
Consider the initial value problem in $\Bbb R^3$ for $\partial_t^2 u - \Delta
u=0$ with initial data
$$u(0, x)= e^{i|x|/\ep}{f(|x|)\over |x|},\hbox{ and } u_t(0, x)=0,$$
where $f(s)\in C_0^\infty(0, \infty)$. Setting $g(s)=f(s)\exp(is/\epsilon)$ for
$s>0$, we extend $g(s)$ to be odd on $\Bbb R$, i.e.
$$g(-|x|)=-f(|x|)\exp(i|x|/\ep).$$ This problem has the exact solution
$$u(t, x)={1\over |x|}(g(t+|x|)-g(t-|x|)).$$
At $x=0$ this solution has a caustic of the
maximum possible strength, since all rays
starting inward from the sphere $|x|=a$ arrive at
$x=0$ when $t=a$. This is reflected in the behavior of the exact solution
$$u(0,t)=g^\prime(t)=(if(t)/\epsilon+f^\prime(t))e^{i t/\epsilon},$$
which grows like $\epsilon^{-1}$ as $\epsilon$ goes to zero.

To build a Gaussian beam approximation for this we need
$$u_{GB}(t, x)={1\over 2}\left({\beta\over 2\pi \ep}\right)^{3/2}\int_{\Bbb
R^3}A^+(t,y)e^{i\Phi^+(t, x; y)/\ep}+A^-(t,y)e^{i\Phi^-(t, x;y)/\ep}dy,$$
where $A^\pm(0,y)=f(|y|)/|y|$ and
$$
\Phi^\pm(0, x; y)=|y|+(x-y)\cdot
 p(y)+(x-y)\cdot\left( \frac{1}{|y|}(I-P(y)+i\beta I\right)(x-y)/2,
$$ where $p(y)=y/|y|$ and $P(y)$
is the orthogonal projection on the span of $p(y)$. We also want
$\Phi^+_t(0,x;y)=-\Phi^-_t(0,x;y)$, so that $\partial_t u(0, x)=0$ and
$A^\pm(t,y)\exp(ik\Phi^\pm(t, x; y))$ must be a lowest order Gaussian
beams concentrated on the null bi-characteristics for $\tau\pm
|\xi|$. With these definitions we have
$$
\Phi^\pm(t, x; y)=|y|+(x-x^\pm(t,y))\cdot p(y)+$$
$$
{1\over
2}(x-x^\pm(t,y))\cdot[i\beta P(y)+{1+i\beta |y| \over |y| \pm
t(1+i\beta|y|)}(I-P(y))](x-x^\pm(t,y)),$$ where $x^\pm(t,x;y)=y\pm tp(y)$.
For the amplitudes we have
$$
A^\pm(t, x\pm t p(y))=(1\pm t(1+i\beta ))^{-1}A(0; x).
$$
Evaluating $u_{GB}(t,x)$ analytically looks
difficult, but for $u_{GB}(t,0)$ one has for
$t>0$
$$u_{GB}(t,0)=u(t,0)+o(1/\epsilon).$$
This behavior is predicted by the basic result
that, like Fourier integral operators, Gaussian
beam superpositions give accurate leading order
terms in asymptotic expansions.

In principle, one can evaluate the Gaussian beam
superposition and compare it with the exact
solution. Doing this numerically could lead to
interesting results on the accuracy of these
superpositions.

\section{Higher order Approximations}
The accuracy of the phase space based Gaussian
beam superposition depends on accuracy of the
individual Gaussian beam {\it Ansatz}. Gaussian
beams can be constructed to satisfy the wave
equation modulo errors of order $\ep^N$, for
arbitrary $N$, by computing higher order terms in
the spatial Taylor series for the phase and
amplitude about the central ray. If we refer the
construction in previous sections as the first
order GB solution, then a $k^{th}$ order GB
solution will include the Taylor series up to
$(k+1)^{th}$ order for the phase, and
$(k-1-2l)^{th}$ order for the $l^{th}$ amplitude
$A_l$ for $l=0, \cdots, \left[
\frac{k-1}{2}\right]$. The equations for these
phase and amplitude Taylor coefficients are
derived recursively, starting with the phase and
then progressing through the amplitudes. At each
stage (phase function, leading amplitude, next
amplitude ...) one has to derive the Taylor
series up to sufficiently high order before
passing to the next function in the expansion.

Let $X=X^\pm(t; X_0)$, with $x=x^\pm(t;X_0)$, denote the
bicharacteristic  at time $t>0$, which originates from $X_0$.
Following \cite{Tan08, LR:2009} we define the $k^{th}$ order
Gaussian beams as follows
\begin{equation}\label{kgb+}
u^\pm_{kGB}(t, y; X_0)=\rho(y-x)\left[ \sum_{l=0}^{\lfloor \frac{k-1}{2}\rfloor }
\ep^l T_{k-1-2l}^{x}[A_l^\pm](y) \right] \exp\left(\frac{i}{\epsilon}T_{k+1}^{x}[\Phi^\pm](y)\right),
\end{equation}
where $T_k^x[f](y)$ is the $k^{th}$ order Taylor polynomial of $f$ about $x$ evaluated at $y$, and $\rho$ is a cut-off function such that on its support the Taylor expansion of $\Phi^\pm$ still has a positive imaginary part.

By invoking the volume preserving map $X=X^\pm(t, X_0)$ and its inverse map denoted by $X_0=X_0^\pm(t, X)$,  we obtain a phase space based $k^{th}$ order Gaussian beam Ansatz
$$
u^\pm_{kPGB}(t, y, X):=u^\pm_{kGB}(t, y; X_0(t, X)).
$$
Proceeding as previously, we form the superpositions.
\begin{equation}\label{kegb+}
u_k^\epsilon(t, y)=Z(n,\epsilon)\left[ \int_{\Omega^+(t)}u^+_{kPGB}\delta(w^+)dX +\int_{\Omega^-(t)}u^-_{kPGB}\delta(w^-)dX \right],
\end{equation}
where $
\Omega(t)=X(t, \Omega(0)),
$
and $w^\pm(t, X)$ is the solution of the Liouville equation with $H=\pm c(x)|p|$
 subject to $w^\pm(0, X)=p-\nabla_x S_{\rm in}(x)$.

 In (\ref{kgb+}) the initial data for the
 amplitudes $A^\pm_l$ must be chosen consistently
 with the initial data (\ref{ini}). This leads to
 the recursion relations
 \begin{align}\label{all}   A_l^+(0, x)+ A_l^-(0, x)&=A_{\rm
 in}^{(l)}\\ \notag
\partial_t A^+_{l-1}(0,x)+ \partial_t A^-_{l-1}&
(0,x) -i(A^+_l(0,x)-A^-_l(0,x))c(x)|\nabla S_{\rm
in}(x)|=B_{\rm in}^{(l-1)}(x).
\end{align}
  Note that, since this recursion involves the
 initial
 time derivatives of the amplitudes, it becomes quite complicated as $l$ increases.

This gives a $k^{th}$ order asymptotic solution of the wave equation. More precisely, we have the following theorem.
\begin{thm}\label{3.3} Let $P$ be the linear wave operator of the form
$P=\partial^2_t -c^2 \Delta $, and $u^\ep$ is defined in (\ref{kegb+})
with $Im(M^\pm_{\rm in})=\beta I$ and $Z(n, \epsilon)=(\beta/(2\pi\ep))^{n/2}$,
$\beta>0$, then $u_k^\epsilon$ is an asymptotic solution and satisfies
\begin{equation}\label{pk+}
   \|P[u_k^\ep](t, \cdot)\|_{L_y^2} \lesssim
\ep^{\frac{k}{2}-1+\frac{1-n}{4}}.
\end{equation}
\end{thm}
\begin{proof}For notational convenience we estimate only one of two Gaussian beams
with $\pm$ index omitted:
\begin{align*}
u_k^\epsilon(t, y)=Z(n,\epsilon)\int_{I(0)}u_{kGB}(t,
y; x_0)dx_0.
\end{align*}
According to the GB construction, $u_{kGB}(t,
y; x_0)$ are asymptotic solutions for each $x_0$, so will be their superpositions $u_k^\epsilon(t, y)$.
It remains to verify (\ref{pk+}).  First we see that
$$
P[u_k^\epsilon(t, y)]=Z(n,\epsilon)\int_{I(0)}P[u_{kGB}(t,
y; x_0)]dx_0.
$$
Using (\ref{pap}) in Section 2 with $A$ replaced by $\rho(y-x)\left[ \sum_{l=0}^{\lfloor \frac{k-1}{2}\rfloor }
\ep^l T_{k-1-2l}^{x}[A_l](y) \right]$ and $\Phi$ by
$T_{k+1}^x[\Phi](y)$, we have

\begin{align*}
c_{-2}(t, y) & = - \tilde G \rho(y-x)
T_{k-1}^x[A_0](y),\\
c_{-1}(t, y) &= 2iL \left[\rho T_{k-1}^x[A_0](y) \right]+\tilde G T_{k-3}^x[A_1](y),\\
c_{l}(t, y)= &2iL \left[\rho T_{k-3-2l}^x[A_{l+1}](y) ]\right] +\tilde G \rho T_{k-5-2l}^x[A_{l+2}](y)+P[\rho T_{k-1-2l}^x[A_{l}](y)], \quad l=0, 1, \cdots,
\end{align*}
where $\tilde G=[(\partial_t T_{k+1}^x[\Phi](y))^2 -c^2 (\nabla_y T_{k+1}^x[\Phi](y))^2]$. Using $T_{k+1}^x[\Phi](y)=\Phi(y)+R_{k+1}^x[\Phi](y)$, here $R^x_{k+1}$ denotes the remainder of the Taylor expansion,  and $\tilde G(t, y)=O(|y-x|^{k+2})$ we can see that
$$
|c_{-2}(t, y)|\leq C|y-x|^{k+2}.
$$
Also using the construction for $A_l$ and their derivatives, we are able to show
$$
|c_{l}(t, y)|\leq C|y-x|^{k-2-2l},
$$
where we have used the fact that differentiation of $\rho$ vanishes in a neighborhood of $x$.   The use of the cut-off function ensures that we can always choose a small neighborhood of $x(t, x_0)$ so that
$$
Im(T_{k+1}^x[\Phi](y))\geq c |y-x|^2.
$$
Consequently,
\begin{align*}
 Z^{-1} \|P[u^\epsilon(t, \cdot)]\|_{L^2} &
 \leq  \left \|\int_{I(0)} A e^{-Im(T_{k+1}^x[\Phi](y))/\epsilon}
\left|\ep^{-2} c_{-2} +\ep^{-1} c_{-1}  + c_0 +\cdots \right| dx_0 \right\|_{L^2_y} \\
& \leq \sum_{j=-2}^{\lfloor \frac{k-1}{2}\rfloor-2}
\ep^{j}  \left \| \int_{I(0)} |A||c_j| e^{-c|y-x(t, x_0)|^2/\epsilon} dx_0 \right\|_{L^2_y},
\end{align*}
continuing the estimate by using  the key estimate (\ref{Tf})
\begin{align*}
&  \lesssim \left[\ep^{-2} \ep^{k/2+1} +\ep^{-1}\cdot \ep ^{k/2}+\cdots  \right]\ep^{(1+n)/4}  \\
&  \lesssim \ep^{k/2-1 +(1+n)/4},
\end{align*}
which when using $Z\sim \ep^{-n/2}$ proves the result.
\end{proof}
In order to obtain an estimate of $\|(u_k^\ep -u)(t, \cdot)\|_E $ for any $ 0\leq t\leq T$, all that remains to verify is that  the superposition (\ref{kegb}) accurately approximates
 the initial data. However, using the recursion (\ref{all}) to determine the amplitudes,
 this is again an application of \cite{Tan08}.
which shows that the initial error in energy norm
is bounded by $\ep^{k/2}$ for $k>1$.  Thus our
main result for $k^{th}$ order phase space GB
superposition is as follows.
\begin{thm}
Given $T>0$, and let $u$ be the solution of the wave equation subject to the initial data $(u, u_t)(0, x)=(A_{\rm in}, B_{\rm in})e^{iS_{\rm in}(x)/\ep}$, and $u^\epsilon$
be the $k^{th}$ order approximation defined in (\ref{kegb+}) with initial data chosen as described above with $|supp(A_{\rm in})|+|supp(B_{\rm in})|< \infty$.
Then there exists $\epsilon_0>0$, a normalization parameter $Z(n,
\epsilon)\sim \ep^{-n/2}$, and a constant $C$ such that for all $\epsilon \in (0,
\epsilon_0)$
$$
\|(u^\epsilon -u)(t, \cdot)\|_E \leq  C \ep^{\frac{k}{2}+\frac{1-n}{4}}
$$
for $t\in [0, T]$.
\end{thm}
\noindent{\bf Remarks}
\begin{itemize}
\item Due to the property of symmetry in time, all results obtained apply to $|t|\leq T$.
\item For higher order constructions, the Liouville equation for higher order GB components can be given similarly to those for the first order GB method.

 \item For computation of high order derivatives of the phase through level set functions we refer to \cite{LR:2009} for details.
\end{itemize}

\bigskip
\section*{Acknowledgments} Liu's research was partially supported by the National Science Foundation under the Kinetic FRG Grant DMS07-57227. He also wants to thank the IPAM for the hospitality and support during his staying for
the program on ``{\it Quantum and Kinetic Transport: Analysis, Computations, and New Applications}",  March 9--June 12, 2009.

\end{document}